\documentclass[11pt, a4paper]{amsart}
\usepackage{amsmath,amssymb,amsthm,mathtools,wasysym,calc,verbatim,tikz,url,hyperref,mathrsfs,cite,fullpage,bbm}
\usepackage[noabbrev,capitalize,nameinlink]{cleveref}
\usepackage[shortlabels]{enumitem}

\usepackage{comment}

\newtheorem{theorem}{Theorem}[section]
\newtheorem{lemma}[theorem]{Lemma}

\newtheorem{prop}[theorem]{Proposition}

\newtheorem{fact}[theorem]{Fact}
\newtheorem{definition}[theorem]{Definition}

\theoremstyle{definition}

\theoremstyle{remark}

\newtheorem*{remark*}{Remark}

\newcommand{\snorm}[1]{\lVert#1\rVert}

\newcommand{\imod}[1]{~\mathrm{mod}~#1}
\newcommand{\eps}{\varepsilon}

\newcommand{\mb}{\mathbb}

\newcommand{\mbm}{\mathbbm}
\newcommand{\mc}{\mathcal}

\newcommand{\mr}{\mathrm}

\newcommand{\ol}{\overline}
\newcommand{\on}{\operatorname}

\newcommand{\wt}{\widetilde}

\begin{document}

\title{Improved Bounds for Szemer\'{e}di's Theorem}

\author[A1]{James Leng}
\address{Department of Mathematics, UCLA, Los Angeles, CA 90095, USA}
\email{jamesleng@math.ucla.edu}

\author[A2]{Ashwin Sah}
\author[A3]{Mehtaab Sawhney}
\address{Department of Mathematics, Massachusetts Institute of Technology, Cambridge, MA 02139, USA}
\email{\{asah,msawhney\}@mit.edu}
\thanks{Leng was supported by NSF Graduate Research Fellowship Grant No.~DGE-2034835. Sah and Sawhney were supported by NSF Graduate Research Fellowship Program DGE-2141064.}

\maketitle
\begin{abstract}
Let $r_k(N)$ denote the size of the largest subset of $[N] = \{1,\ldots,N\}$ with no $k$-term arithmetic progression. We show that for $k\ge 5$, there exists $c_k>0$ such that
\[r_k(N)\ll N\exp(-(\log\log N)^{c_k}).\]
Our proof is a consequence of recent quasipolynomial bounds on the inverse theorem for the Gowers $U^k$-norm as well as the density increment strategy of Heath-Brown and Szemer\'{e}di as reformulated by Green and Tao. 
\end{abstract}

\section{Introduction}\label{sec:introduction}
Let $[N] = \{1,\ldots, N\}$ and $r_k(N)$ denote the size of the largest $S\subseteq[N]$ such that $S$ has no $k$-term arithmetic progressions. The first nontrivial upper bound on $r_3(N)$ came from work of Roth \cite{Rot54} which proved
\[r_3(N)\ll N(\log\log N)^{-1}.\]
A long series of works improved this bound, including works of Heath-Brown \cite{Hea87}, Szemer{\'e}di \cite{Sze90}, Bourgain \cite{Bou99,Bou08}, Sanders \cite{San12, San11}, Bloom \cite{Blo16}, and Bloom and Sisask \cite{BS20}. In breakthrough work, Kelley and Meka \cite{KM23} very recently proved 
\[r_3(N)\ll N\exp(-c(\log N)^{1/12});\]
the constant $1/12$ was refined to $1/9$ in work of Bloom and Sisask \cite{BS23}.

For higher $k$, a long-standing conjecture of Erd\H{o}s and Tur\'{a}n stated that $r_k(N) = o(N)$. In seminal works, Szemer\'{e}di \cite{Sze70,Sze75} first established the estimate $r_4(N) = o(N)$
and then established his eponymous theorem that
\[r_k(N) = o(N).\]
Due to uses of van der Waerden theorem and the regularity lemma (which was introduced in this work), Szemer\'{e}di's density saving was exceedingly small. In breakthrough work, Gowers \cite{Gow98,Gow01} introduced higher order Fourier analysis and proved the first ``reasonable'' bounds for Szemer\'{e}di's theorem:
\[r_k(N) < N(\log\log N)^{-2^{-2^{k+9}}}.\]
The only improvement to this result for $k\ge 4$ was work of Green and Tao \cite{GT09,GT17} which ultimately established that 
\[r_4(N)\ll N(\log N)^{-c},\]
and recent work of the authors \cite{LSS24} which proved 
\[r_5(N)\ll N\exp(-(\log\log N)^c).\]

Our main result is an extension of this bound for all $k\ge 5$.
\begin{theorem}\label{thm:main}
Fix $k\ge 5$. There is $c_k\in(0,1)$ such that
\[r_k(N)\ll N\exp(-(\log\log N)^{c_k}).\]
\end{theorem}

\subsection{Proof outline and techniques}\label{sub:outline}
\subsubsection{Local and global inverse theorems}

The primary input to our result will be the main result of recent work of the authors \cite{LSS24b}, i.e., quasipolynomial bounds on the inverse theorem for the Gowers $U^{k+1}$-norm. Given an inverse theorem, the deduction of Szemer\'{e}di's theorem via a standard density increment strategy is essentially folklore and was recorded in work of Green and Tao \cite{GT10c} (although, prior to \cite{LSS24b} the resulting bounds would be far from matching those of Gowers \cite{Gow01}). However, if one naively follows this script using \cite{LSS24b}, one obtains a bound of $N\exp(-(\log\log\log N)^{-\Omega_k(1)})$ which is weaker than the work of Gowers \cite{Gow01}. Furthermore, Gowers's argument makes use of a ``local'' inverse theorem that in fact gives a slightly stronger correlation compared to the bound given for the ``global'' inverse theorem in \cite{LSS24b} (namely, polynomial versus quasipolynomial). Thus, this global nature of \cite{LSS24b} must be exploited. Additionally, use of global inverse theorems necessitates understanding of nilsequences and polynomial sequences on nilpotent Lie groups, as opposed to merely polynomials as in the work of Gowers \cite{Gow01}.

\subsubsection{Schmidt-type decomposition problems}
This is done via the improved density increment strategy of Heath-Brown \cite{Hea87} and Szemer\'{e}di \cite{Sze90} which involves extracting a set of functions to correlate with instead of simply one and using this to give a multiplicative density increment. Such a strategy was given a robust formulation in work of Green and Tao \cite{GT09} on four-term progressions; in particular, their reformulation avoided the explicit Fourier-analytic formulas used in \cite{Hea87,Sze90} and thus is applicable to the higher order setting. The strategy here runs smoothly given the inverse theorem, modulo resolving a certain Schmidt-type problem for nilsequences. In particular, given a polynomial sequence $g(n)$ with $g(0) = \mr{id}_{G}$ on a nilmanifold $G/\Gamma$ of degree $k$ with complexity $M$ and dimension $d$, one needs to prove that 
\[\min_{1\le n\le N}d_{G/\Gamma}(\mr{id}_G,g(n)\Gamma)\ll M^{O_k(d^{O_k(1)})}N^{-1/d^{O_k(1)}}.\]
In particular, the polynomial dependence on dimension within the exponent is key.

We in fact require a certain slightly stronger result (decomposing $[N]$ into long arithmetic progressions $P$ such that the diameters of the sets $\{g(n)\Gamma\colon n\in P\}$ are small), which is the heart of the matter for this work. When the underlying nilpotent group $G$ is abelian, this is easily deduced from a result of Schmidt \cite{Sch77} (see Lemma~\ref{lem:schmidt-gen}, or \cite[Section~6]{GT09} in the quadratic case).

For general degree $2$ nilmanifolds such a problem was implicitly solved in work of Green and Tao \cite{GT09} and for degree $3$ nilmanifolds it was essentially solved in recent work of the authors \cite{LSS24}. More precisely, \cite{LSS24} essentially proves that given a list of bracket expressions $(a_in \lfloor b_i n\rfloor \lfloor c_i n\rfloor)_{1\le i\le d}$ that
\[\min_{1\le n\le N}\snorm{a_in \lfloor b_i n\rfloor \lfloor c_i n\rfloor}_{\mb{R}/\mb{Z}}\le N^{-1/d^{O(1)}}\]
and via an explicit computation with fundamental domains on degree $3$ nilmanifolds one may reduce to such a situation. The proof given in \cite{LSS24} relies on the fact that $3$ is sufficiently small and in particular that it is possible to reduce to a situation in which there are no ``nested integer part operations'' as one attempts to solve the ``bracket Schmidt'' problem in one go.

\subsubsection{Iterative Schmidt refinement}\label{sub:iterative-schmidt}
The key observation required for our work, at least at a heuristic level, is a procedure for solving such ``bracket Schmidt'' problems even when there are nested brackets. As a simple example, consider bracket expressions $(a_in \lfloor b_i n\lfloor c_i n\rfloor\rfloor)_{1\le i\le d}$. We will solve the Schmidt problem via iteratively ``reducing'' the number of brackets from the inside-out (at the cost of passing to subprogressions). In particular, using Dirichlet's theorem, one can break $[N]$ into arithmetic progressions $P$ each of length $N^{1/d^{O(1)}}$ such that when restricted to each arithmetic progression, every function $\lfloor c_i n\rfloor$ is a linear function (i.e., it is a \emph{locally linear function} on each $P$). Since the only locally linear functions on a progression agree with genuinely linear functions, we can replace $\lfloor c_i n\rfloor$ by $d_{i,P}n + e_{i,P}$ and reduce to considering the bracket expression $a_in \lfloor b_i n (d_{i,P} n + e_{i,P})\rfloor$ when restricted to $P$. One can then iterate this argument on the ``inner quadratics'' $b_i n (d_{i,P} n + e_{i,P})$ (essentially using abelian Schmidt as discussed above for degree $2$ in this case). We may find a decomposition into long arithmetic progressions $Q$ such that $\lfloor b_i n (d_{i,P} n + e_{i,P})\rfloor$ is locally quadratic (and hence agrees with a global quadratic) on each $Q$. Thus, restricted to any such $Q$, our original functions $a_in\lfloor b_in\lfloor c_in\rfloor\rfloor$ agrees with a genuine cubic. Finally, we can decompose these progressions $Q$ into ones where the cubics are approximately constant $\imod{1}$ (using abelian Schmidt for degree $3$). While in theory this approach can be made to work for all such bracket Schmidt problems, this however necessitates working with bracket functions and rather quickly becomes messy to handle. 

This procedure can be adapted to work with polynomial sequences on nilmanifolds directly due to an unpublished observation of Green and Tao. This is the approach we take in the present work. The crucial point is that given a polynomial sequence $g(n)$ with respect to a group $G$ given a filtration $G_0 = G_1 \geqslant G_2 \geqslant \cdots \geqslant G_k \geqslant \mr{Id}_G$, the polynomial sequence $g(n) \imod G_2$ is a standard polynomial. Thus one can apply Schmidt to a standard polynomial and therefore (after passing to long subprogressions) one may factor $g(n) \imod G_2 = \eps(n) \cdot \gamma(n)$ where $\eps$ is smooth and $\gamma$ lies in the lattice $\Gamma \imod G_2$. One may then lift $\eps,\gamma$ from $G\imod G_2$ to $\wt{\eps},\wt{\gamma}$ on $G$ and analyze the polynomial sequence $\wt{\eps}^{-1} g \wt{\gamma}^{-1}$ which now lives in the group $G_2$. One can iterate this procedure and inductively reduce $G_2$ to $G_3$ and so on, which allows us to solve the Schmidt problem for our nilmanifold. We remark that this procedure is an induction on the length of the filtration whereas the (closely related) approach taken in \cite{GT10c} is phrased as an induction on dimension. This difference is crucial for getting bounds in which the exponent depends polynomially on dimension.

\subsection{Organization and notation}\label{sub:notation}
All definitions regarding nilsequences and associated complexity will be exactly as in \cite[Sections~3--4]{LSS24b}. We refer the reader to that paper for all such definitions; we will only require degree filtrations in this paper. 

We use standard asymptotic notation. Given functions $f=f(n)$ and $g=g(n)$, we write $f=O(g)$, $f \ll g$, $g=\Omega(f)$, or $g\gg f$ to mean that there is a constant $C$ such that $|f(n)|\le Cg(n)$ for sufficiently large $n$. We write $f\asymp g$ or $f=\Theta(g)$ to mean that $f\ll g$ and $g\ll f$, and write $f=o(g)$ or $g=\omega(f)$ to mean $f(n)/g(n)\to0$ as $n\to\infty$. Subscripts on asymptotic notation indicate dependence of the bounds on those parameters. We will use the notation $[x] = \{1,2\ldots,\lfloor x\rfloor\}$. In this paper $x = \lfloor x\rfloor + \{x\}$ where $\{x\}\in [0,1)$ and $\lfloor x\rfloor \in \mb{Z}$; we remark this is different than in \cite{LSS24b}. We write $\snorm{x}_{\mb{R}/\mb{Z}}=\on{dist}(x,\mb{Z})$ for $x\in\mb{R}$. Furthermore throughout this paper we abusively write $\log$ for $\max(\log(\cdot), e^e)$; this is to avoid trivial issues with small numbers.

Finally, in terms of organization, in Section~\ref{sec:schmidt} we solve the Schmidt problem for nilsequences and in Section~\ref{sec:proof} we prove Theorem~\ref{thm:main}.

\subsection*{Acknowledgments}
The third author thanks Mark Sellke and Dmitrii Zakharov for helpful and motivating conversations. We thank Ben Green for helpful comments on the manuscript. We thank Zach Hunter for various minor corrections.

\section{Schmidt's problem for nilsequences}\label{sec:schmidt}
In this section, we prove that given a list of nilsequences on $[N]$, one can decompose $[N]$ into a controlled set of arithmetic progressions such that the nilsequences are almost constant on these sequences.

\begin{lemma}\label{lem:schmidt-nilman}
Consider nilmanifolds $G_i/\Gamma_i$ for $1\le i\le T$, each given a degree $k$ filtration, having complexity bounded by $M$, dimension bounded by $d$, and for each $1\le i\le T$ let $g_i(n)$ be a polynomial sequence with respect to the specified degree $k$-filtration on $G_i$.

We may decompose $[N]$ into disjoint arithmetic progressions $\mc{P}_1,\ldots, \mc{P}_L$ such that following conditions hold:
\begin{itemize}
    \item $N/L\ge N^{\Omega_k(1/(Td)^{O_k(1)})}/2$;
    \item We have
    \[\max_{\substack{1\le i\le T\\1\le j\le L}}\max_{n,n'\in\mc{P}_j}d_{G_i/\Gamma_i}(g_i(n)\Gamma_i,g_i(n')\Gamma_i)\le M^{O_k(d^{O_k(1)})}\cdot N^{-\Omega_k(1/(Td)^{O_k(1)})}.\]
\end{itemize}
\end{lemma}

The key ingredient in this proof is a result of Schmidt \cite{Sch77} regarding finding small fractional parts of polynomials. We will need a version of this result with explicit quantification; this is explicitly stated in work of the authors \cite[Proposition~3.7]{LSS24} although the argument is essentially verbatim from a paper of Green and Tao \cite[Appendix~A]{GT09} generalized from quadratics to all degrees.

\begin{prop}\label{prop:schmidt}
Fix an integer $k\ge 1$. There exist $c_k>0$ such that the following holds. Let $\alpha_1,\ldots,\alpha_d$ be real numbers. Then 
\[\min_{1\le n\le N}\max_{1\le i\le d}\snorm{\alpha_i n^k}_{\mb{R}/\mb{Z}}\ll_k dN^{-c_k/d^2}.\]
\end{prop}

As stated this result is for pure monomial phases and only provides a single point with small fractional part. This statement however can be ``upgraded'' via a straightforward iterative argument which is implicit in say \cite[Proposition~6.4]{GT09} (where the quadratic case is handled).

\begin{lemma}\label{lem:schmidt-gen}
Fix an integer $k\ge 0$. Consider polynomials $Q_1,\ldots,Q_d$ of degree $k$. Then there exist disjoint arithmetic progressions $\mc{P}_1,\ldots,\mc{P}_L$ such that following conditions hold:
\begin{itemize}
    \item $N/L\ge N^{\Omega_k(1/d^{O_k(1)})}/2$
    \item We have
    \[\max_{\substack{1\le i\le d\\1\le j\le L}}\max_{n,n'\in \mc{P}_j}\snorm{Q_i(n) - Q_i(n')}_{\mb{R}/\mb{Z}}\le 2\cdot N^{-\Omega_k(1/d^{O_k(1)})}.\]
\end{itemize} 
\end{lemma}
\begin{proof}
We proceed by induction on $k$. The case $k = 0$ is trivial as $Q_j(\cdot)$ are constant. Furthermore we may assume that $N\ge\exp(d^{\Omega_k(1)})$ else we may break $[N]$ into singleton arithmetic progressions.

Let $Q_j(n) = \sum_{\ell=0}^{k}\alpha_{j,\ell}n^{\ell}$. Applying Proposition~\ref{prop:schmidt}, there exists $D\le N^{1/2}$ such that 
\[\max_{1\le j\le d}\snorm{\alpha_{j,k} D^k}_{\mb{R}/\mb{Z}}\ll_k dN^{-c_k/(2d^2)} =: \tau.\]
We break $[N]$ into arithmetic progressions of common difference $D$ and with lengths between $2^{-1}\tau^{-1/(2k)}$ and $\tau^{-1/(2k)}$. Label these progressions $\mc{R}_1,\ldots,\mc{R}_{L'}$ with starting points $s_{i}$ for $1\le i\le L'$. We have 
\[Q_j(Dn+s_i) = \alpha_{j,k}D^{k}n^{k} + Q_{j,i}(n)\]
for appropriately defined polynomials $Q_{j,i}(n)$ of degree at most $k-1$. Note that for $n,n'\in [\tau^{-1/(2k)}]$, we have 
\begin{align*}
\snorm{Q_j(Dn+s_i) - Q_j(Dn' + s_i)}_{\mb{R}/\mb{Z}} &= \snorm{\alpha_{j,k}D^{k}(n^{k}-(n')^k) + Q_{j,i}(n) - Q_{j,i}(n')}_{\mb{R}/\mb{Z}}\\
&\le 2\tau^{-1/2} \cdot \snorm{\alpha_{j,k}D^{k}}_{\mb{R}/\mb{Z}} + \snorm{Q_{j,i}(n) - Q_{j,i}(n')}_{\mb{R}/\mb{Z}}\\
&\le 2 \tau^{1/2} + \snorm{Q_{j,i}(n) - Q_{j,i}(n')}_{\mb{R}/\mb{Z}}.
\end{align*}
The result now follows by induction applied to each $Q_{j,i}(n)$ for $1\le i\le L'$ on the interval $[\tau^{-1/(2k)}]$ and using these decompositions to split the $\mc{R}_i$ into our final decomposition. Letting $N'=\tau^{-1/(2k)}$, the number of arithmetic progressions resulting is bounded by
\[(2N/N')\cdot 2(N')^{1-c_1/d^{c_2}}\le 2N^{1-\Omega_k(1/d^{O_k(1)})},\]
where $c_1,c_2$ are the implicit constants for the inductive hypothesis $k-1$. The result follows.
\end{proof}

We next require the following lemma controlling coefficients of polynomials which live in a restricted range $\imod{1}$. It will be convenient to recall the smoothness norm of a polynomial $P(n) = \sum_{i=0}^{k}\alpha_i \binom{n}{i}$ which is defined as 
\[\snorm{P}_{C^{\infty}[N]} := \max_{1\le i\le k} N^i\snorm{\alpha_i}_{\mb{R}/\mb{Z}}.\]
\begin{lemma}\label{lem:small+int}
Fix an integer $k\ge 1$. There exists $c_k>0$ such that if $\eps\in (0,c_k)$ and $N\ge c_k^{-1}$ then the following holds. Consider a polynomial $P(n) = \sum_{i=0}^{k}\alpha_i \binom{n}{i}$. Suppose that for $n,n'\in [N]$, we have $\snorm{P(n)-P(n')}_{\mb{R}/\mb{Z}}\le \eps$. Then
\[\snorm{P}_{C^{\infty}[N]}\ll_{k} \eps.\]
\end{lemma}
\begin{proof}
Note that
\[\bigg|\sum_{n=1}^N e(P(n))\bigg| \ge N/2.\]
By a quantitative version of Weyl's inequality, which may be found in Green and Tao \cite[Proposition~4.3]{GT12}, there exists $q\in \mb{N}$ with $q\ll_k 1$ such that 
\[\snorm{qP}_{C^{\infty}[N]}\ll_{k} 1.\]
Let $1\le t \le \lfloor N/(2k)\rfloor$ be an integer and note that 
\[\alpha_k \cdot t^{k} = \sum_{i=0}^{k}(-1)^{k-i}\binom{k}{i}\cdot P(t\cdot i + 1).\]
Via the triangle inequality, we therefore have 
\[\snorm{\alpha_k \cdot t^{k}}_{\mb{R}/\mb{Z}}\le 2^{k-1}\eps.\]
Take $t$ to be a prime between $\lfloor N/(2C)\rfloor$ and $\lfloor N/C\rfloor$ where $C$ is a sufficiently large absolute constant in terms of $k$. Combining this with the estimate $\snorm{q\alpha_k}_{\mb{R}/\mb{Z}}\ll_{k} N^{-k}$ implies that $\snorm{\alpha_k}_{\mb{R}/\mb{Z}}\ll_{k} \eps \cdot N^{-k}$. The result then follows by induction on $k$ and applying the result for the degree $(k-1)$ polynomial $P'(n) = \sum_{i=0}^{k-1}\alpha_i\binom{n}{i}$.
\end{proof}

With this we are in position to deduce the result for nilsequences along the lines sketched in Section~\ref{sub:iterative-schmidt}.

\begin{proof}[Proof of Lemma~\ref{lem:schmidt-nilman}]
Consider the degree $k$ filtration of the group $G_i$, $G_{i,0} = G_{i,1} \geqslant G_{i,2} \geqslant \cdots \geqslant G_{i,k}\geqslant\mr{Id}_{G_i}$. We say the group $G_i$ has a degree $k$ filtration of type $t$ if $G_{i,t} = G_i$ (i.e., the first $(t+1)$ groups in the filtration match). We prove the result by backwards induction on $t$ assuming that all groups $G_i$ have degree $k$ filtrations of type $t$; note that the result is trivial when $t = k+1$ and we aim to prove the claim when $t = 1$. So, consider the case where the filtration has type $t$ for some $1\le t\le k$ and suppose that we already know cases of larger type.

Let $\mc{X}_{i} = \{X_{i,1},\ldots,X_{i,\dim(G_i)}\}$ denote the Mal'cev basis for $G_i$. By the classification of polynomial sequences (see \cite[Lemma~6.7]{GT10b}), we have 
\[g_i(n) = \exp\bigg(\sum_{j=1}^{\dim(G_i)}P_{i,j}(n) \cdot X_{i,j}\bigg)\]
where if $X_{i,j}\in (\mc{X}_i\cap \log(G_{i,\ell}))\setminus (\mc{X}_i\cap \log(G_{i,\ell+1}))$ then the degree of polynomial $P_{i,j}(n)$ is bounded by $\ell$. 

We consider the polynomials $P_{i,j}(n)$ for $1\le i\le T$ and $1\le j \le \dim(G_i) - \dim(G_{i,t+1})$. The degrees of $P_{i,j}(n)$ are all at most $t\le k$ and the total number of polynomials number consideration is bounded by $T\cdot d$. By applying Lemma~\ref{lem:schmidt-gen}, there exists a decomposition of $[N]$ into arithmetic progressions $\mc{P}_{1},\ldots,\mc{P}_L$ such that:
\begin{itemize}
    \item $N/L\ge N^{\Omega_k(1/(dT)^{O_k(1)})}/2$
    \item We have
    \[\max_{\substack{1\le i\le T\\1\le j \le \dim(G_i) - \dim(G_{i,t+1})}}\max_{1\le s\le L}\max_{n,n'\in\mc{P}_s}\snorm{P_{i,j}(n) - P_{i,j}(n')}_{\mb{R}/\mb{Z}}\le 2\cdot N^{-\Omega_k(1/(dT)^{O_k(1)})}.\]
\end{itemize}

We break the progressions $\mc{P}_s$ into two classes: the first class ($s\in\mc{S}$) if the progression has length bounded by $\sqrt{N/L}$ and the second class ($s\in\mc{L}$) otherwise. For progressions which are short, we break each such progression into singletons; after this there are at most $L + \sqrt{N/L} \cdot L \le 2\sqrt{NL}$ progressions which is qualitatively identical to before. For each $s\in\mc{L}$, we write $\mc{P}_s = \{a_sn + b_s\}_{n\in[|\mc{P}_s|]}$ where $|\mc{P}_s|$ denotes the length of the progression. 

Using the second condition above and applying Lemma~\ref{lem:small+int}, we see that for each long progression $\mc{P}_s$, we have for all $i,j$ that
\[P_{i,j}(a_sn + b_s) = P_{i,j,s,\mr{small}}(n) +  P_{i,j,s,\mr{int}}(n)\]
where:
\begin{itemize}
    \item $\deg(P_{i,j,s,\mr{int}}),\deg(P_{i,j,s,\mr{small}})\le\deg(P_{i,j})$
    \item $P_{i,j,s,\mr{int}}$ maps $\mb{Z}\to\mb{Z}$
    \item If $P_{i,j,s,\mr{small}}(n) = \sum_{r=0}^{t}\alpha_{i,j,s,\mr{small},r}\binom{n}{r}$ then 
    \[|\alpha_{i,j,s,\mr{small},r}|\le 2 N^{-r} \cdot N^{-\Omega_k(1/(dT)^{O_k(1)})}\]
    for $1\le r\le t$ and $|\alpha_{i,j,s,\mr{small},0}|\le 1$.
\end{itemize}
We have implicitly used $|\mc{P}_s|\ge\sqrt{N/L}\ge N^{\Omega_k(1/(dT)^{O_k(1)}})$ for $s\in\mc{L}$ here.

The key trick is to now ``reduce'' the polynomial sequence $g_i$ to one which lives in $G_{i,t+1}$. Define
\begin{itemize}
    \item $\eps_{i,s}(n) = \exp\Big(\sum_{j=1}^{\dim(G_i) - \dim(G_{i,t+1})}P_{i,j,s,\mr{small}}(n) \cdot X_{i,j}\Big)$
    \item $\gamma_{i,s}(n) = \prod_{j=1}^{\dim(G_i) - \dim(G_{i,t+1})}\exp(X_{i,j})^{P_{i,j,s,\mr{int}}(n)}$
    \item $g_{i,s}'(n) = \eps_{i,s}(n)^{-1} \cdot g(a_s n + b_s) \cdot \gamma_{i,s}(n)^{-1}$ 
\end{itemize}
Note that $\eps_{i,s}$, $\gamma_{i,s}$ are polynomial sequences with respect to the filtration given on $G_i$ by the classification of polynomial sequences (see \cite[Lemma~6.7]{GT10b}) and the fact that the set of polynomial sequences form a group. Therefore $g_{i,s}'$ is also seen to be a polynomial sequence. The crucial point, however, is that by the Baker--Campbell--Hausdorff formula, we have that $g_{i,s}'$ only takes on values in $G_{i,t+1}$. (We are using the assumption on type that $G_i=G_{0,i}=G_{t,i}$, so any commutator is in $G_{2t,i}\leqslant G_{t+1,i}$ since $t\ge 1$.) 

Therefore we may inductively apply the claim for each long progression $\mc{P}_s$, to the polynomials $g_{i,s}'$ on $G_{i,t+1}$ where we take the filtration on $G_i$ intersected with $G_{i,t+1}$ (note that the filtration is still degree $k$). The corresponding Mal'cev basis is given by taking the last $\dim(G_{i,t+1})$ elements of $\mc{X}_i$. By induction therefore we may break each long $\mc{P}_s$ into $L_s$ such progressions $\mc{P}_{s,r}$ where $L_s\le |\mc{P}_s|^{1-\Omega_k(1/(Td)^{O_k(1)})}$ and such that 
\[\max_{\substack{s\in\mc{L}\\1\le r\le L_s}}\max_{n,n'\in\mc{P}_{s,r}}d_{G_i/\Gamma_i}(g_{i,s}'(n)\Gamma_i,g_{i,s}'(n')\Gamma_i)\le M^{O_k(d^{O_k(1)})}\cdot N^{-\Omega_k(1/d^{O_k(1)})}.\]
Here we are using \cite[Lemma~B.9]{Len23b} to compare distances between $G_i$ and $G_{i,t+1}$. 

Furthermore note that $\gamma_{i,s}$ takes values only in $\Gamma$ by the definition of a Mal'cev basis and that for $n,n'\in [|\mc{P}_s|]$ we have 
\[d_{G_i}(\eps_{i,s}(n),\mr{id}_{G_i})\le M^{O_k(d^{O_k(1)})}\text{ and } d_{G_i}(\eps_{i,s}(n),\eps_{i,s}(n'))\le M^{O_k(d^{O_k(1)})} \cdot N^{-\Omega_k(1/(dT)^{O_k(1)})}.\]
This is due to our bounds on the smoothness norm of $P_{i,j,s,\mr{small}}$ and \cite[Lemma~B.3]{Len23b}.

It therefore follows by \cite[Lemma~B.4]{Len23b} that for any $s,r$ we have
\begin{align*}
\max_{n,n'\in\mc{P}_{s,r}}&d_{G_i/\Gamma_i}(g_i(a_sn + b_s)\Gamma_i,g_i(a_sn' + b_s)\Gamma_i)\\
&=\max_{n,n'\in \mc{P}_{s,r}}d_{G_i/\Gamma_i}(\eps_{i,s}(n) g_{i,s}'(n) \Gamma_i,\eps_{i,s}(n') g_{i,s}'(n') \Gamma_i)\\
&\le\max_{n,n'\in \mc{P}_{s,r}}d_{G_i/\Gamma_i}(\eps_{i,s}(n) g_{i,s}'(n) \Gamma_i,\eps_{i,s}(n) g_{i,s}'(n') \Gamma_i)\\
&\qquad\qquad+ \max_{n,n'\in \mc{P}_{s,r}}d_{G_i/\Gamma_i}(\eps_{i,s}(n) g_{i,s}'(n') \Gamma_i,\eps_{i,s}(n') g_{i,s}'(n') \Gamma_i) \\
&\le M^{O_k(d^{O_k(1)})}\big(\max_{n,n'\in \mc{P}_{s,r}}d_{G_i/\Gamma_i}(g_{i,s}'(n) \Gamma_i, g_{i,s}'(n') \Gamma_i) + \max_{n,n'\in \mc{P}_{s,r}}d_{G_i}(\eps_{i,s}(n),\eps_{i,s}(n'))\big) \\
&\le M^{O_k(d^{O_k(1)})} \cdot N^{-\Omega_k(1/(dT)^{O_k(1)})}
\end{align*}
which completes the inductive step (our final decomposition is composed of all elements of the short $\mc{P}_s$ indexed by $s\in\mc{S}$ and all $\mc{P}_{s,r}$ arising from the long progressions indexed by $s\in\mc{L}$). We are done, noting that the number of inductive steps (hence the decay in parameters) is bounded in terms of $k$.
\end{proof}

\section{Completing the proof}\label{sec:proof}
We are now run the Heath-Brown \cite{Hea87} and Szemer\'{e}di \cite{Sze90} density increment strategy as reformulated by Green and Tao \cite{GT09}. In the first subsection we recall a number of preliminaries for the proof and in the second subsection we prove Theorem~\ref{thm:main}. Our treatment at this point is quite close to that of \cite{GT09} and we borrow certain elements from the density increment portion of \cite{PP22} as well.

\subsection{Preliminaries for density increment}\label{sub:preliminaries}
We first recall the definition of the Gowers $U^s$-norm over the integers. 

\begin{definition}\label{def:gowers-norm}
Given $f\colon\mb{Z}/N\mb{Z}\to\mb{C}$ and $s\ge 1$, we define
\[\snorm{f}_{U^s(\mb{Z}/N\mb{Z})}^{2^s}=\mb{E}_{x,h_1,\ldots,h_s\in\mb{Z}/N\mb{Z}}\Delta_{h_1,\ldots,h_s}f(x)\]
where $\Delta_hf(x)=f(x)\ol{f(x+h)}$ is the multiplicative discrete derivative (extended to vectors $h$ in the natural way).

Given a natural number $N$ and a function $f\colon[N]\to \mb{C}$, we choose a number $\wt{N}\ge 2^{s}N$ and define $\wt{f}\colon\mb{Z}/\wt{N}\mb{Z}\to\mb{C}$ via $\wt{f}(x) = f(x)$ for $x\in[N]$ and $0$ otherwise. Then
\[\snorm{f}_{U^s[N]} := \snorm{\wt{f}}_{U^s(\mb{Z}/\wt{N}\mb{Z})}/\snorm{\mbm{1}_{[N]}}_{U^s(\mb{Z}/\wt{N}\mb{Z})}.\]
One can check that this definition does not depend on the choice of $\wt{N}$. This is well known to be a seminorm for $s\ge 1$ and a norm for $s\ge 2$.
\end{definition}

As mentioned, the main input for our result will be the following improved bound for the $U^s$-norm inverse theorem given as \cite[Theorem~1.2]{LSS24b}.
\begin{theorem}\label{thm:main-inv}
Fix $\delta\in (0,1/2)$. Suppose that $f\colon[N]\to\mb{C}$ is $1$-bounded and
\[\snorm{f}_{U^{s+1}[N]}\ge\delta.\]
Then there exists a nilmanifold $G/\Gamma$ of degree $s$, complexity at most $M$, and dimension at most $d$ as well as a function $F$ on $G/\Gamma$ which is at most $K$-Lipschitz such that
\[|\mb{E}_{n\in[N]}[f(n)\ol{F(g(n)\Gamma)}]|\ge\eps,\]
where we may take
\[d\le\log(1/\delta)^{O_s(1)}\emph{ and }\eps^{-1},K,M\le \exp(\log(1/\delta)^{O_s(1)}).\]    
\end{theorem}

We now define the $k$-fold linear operator corresponding to counting $k$-term arithmetic progressions. Given functions $f_i\colon[N]\to\mb{C}$, define
\[\Lambda_k(f_1,\ldots,f_k) = \mb{E}_{x,y\in\{0,\ldots,N\}}\prod_{j=1}^{k}f_j(x + (j-1)y)\]
where $f_i$ are extended by $0$ outside of $[N]$. We also write 
\[\Lambda_k(f) :=\Lambda_k(f,\ldots,f).\] We have the following basic inequalities regarding the operator $\Lambda_k$. The proof is by now standard and hence is omitted (see \cite[Lemma~3.2]{GT09} and \cite[Theorem~3.2]{Gow01a}).
\begin{lemma}\label{lem:basic-fct}
Consider functions $f_i\colon[N]\to \mb{C}$ for $1\le i\le k$. Then we have
\begin{align*}
\Lambda_k(f_1,\ldots,f_{k})&\le \min_{1\le i\le k}\snorm{f_i}_{L^1[N]}\cdot \prod_{j\neq i}\snorm{f_j}_{L^\infty[N]},\\
\Lambda_k(f_1,\ldots,f_{k})&\ll_k \min_{1\le i\le k}\snorm{f_i}_{U^{k-1}[N]}\cdot\prod_{j\neq i}\snorm{f_j}_{L^\infty[N]}.
\end{align*}
\end{lemma}

We next define factors and the factor induced by function $g$ with a resolution $K$.
\begin{definition}\label{def:factor}
We define a \emph{factor} $\mc{B}$ of $[N]$ to be a partition $[N] = \bigsqcup_{B\in \mc{B}}B$. We define $\mc{B}(x)$ for $x\in[N]$ to be the part of $\mc{B}$ that contains $x$. We say $\mc{B}'$ refines $\mc{B}$ if every part of $\mc{B}$ can be written as a disjoint union of parts of $\mc{B}'$. We define a \emph{join} of a sequence of factors to be the partition (discarding empty parts)
\[\mc{B}_1\vee\cdots\vee\mc{B}_d := \{B_1\cap\cdots\cap B_d\colon B_i\in \mc{B}_i\}.\]
Next given a function $g\colon[N]\to\mb{R}$ and a resolution $K$, we define the factor induced by $g$ of resolution $K$ to be 
\[\mc{B}_{g,K} = \bigsqcup_{j\in\mb{Z}}\{x\in[N]\colon g(x)\in [j/K, (j+1)/K).\]
Finally, given a factor $\mc{B}$, we define $\Pi_{\mc{B}}f$ by 
\[\Pi_{\mc{B}}f(x) = \mb{E}_{y\in\mc{B}(x)}f(y).\]
\end{definition}

A technical annoyance is that one may potentially have a large set of points near the cutoffs when defining $\mc{B}_{g,K}$. We define a notion of regularity capturing when a function $g$ avoids such issues, which is related to an idea introduced by Bourgain \cite{Bou99} with regards to Bohr sets.
\begin{definition}\label{def:regular}
The factor $\mc{B}_{g,K}$ is $C$-regular if 
\[\sup_{r>0}\bigg(\frac{1}{2r}\frac{1}{N}\big|\{x\in [N]\colon\snorm{K\cdot  g(x)}_{\mb{R}/\mb{Z}}\le r\big\}|\bigg)\le C.\]
\end{definition}

It turns out to be easy to obtain ``regular'' factors; a useful trick (motivated by the proof of \cite[Corollary 2.3]{GT10b}) is to consider a random shift of $g$ and then apply the Hardy–-Littlewood maximal inequality. Given a function $g$ and resolution $K$, we define the maximal function 
\[M_{g,K}(t) := \sup_{r>0}\frac{1}{2r}\frac{1}{N}\big|\{x\in[N]\colon\snorm{K\cdot g(x) - t}_{\mb{R}/\mb{Z}}\le r\}\big|.\]
The Hardy--Littlewood maximal inequality (on the torus $\mb{R}/\mb{Z}$) implies that 
\[\mb{E}_{t\in [0,1]}[M_{g,K}(t)] = O(1).\]
Therefore we have the following elementary fact which will prove useful.

\begin{fact}\label{fct:maximal}
There exists a constant $C = C_{\ref{fct:maximal}}>0$ such that the following holds. Given a function $g\colon[N]\to\mb{R}$ and a resolution $K$, there exists a shift $t\in[0,1/K)$ such that $\mc{B}_{g-t,K}$ is $C$-regular.
\end{fact}

\subsection{Constructing factor approximation and density increment}\label{sub:factor-increment}
The key claim which we need to prove Theorem~\ref{thm:main} is the following density increment lemma, phrased as a trichotomy.
\begin{lemma}\label{lem:dens-increm}
Fix an integer $k\ge 5$ and a constant $c>0$. Consider a function $f\colon[N]\to [0,1]$ such that $\mb{E}_{n\in[N]}f(n) = \delta$. There exist $c' = c'(c,k)$ and $C = C(c,k)$ such that one of the following always holds: 
\begin{itemize}
    \item $N\le \exp(\exp(\log(1/\delta)^{C}))$;
    \item $|\Lambda_k(f) - \Lambda_k(\delta \cdot \mbm{1}_{[N]})|\le c \delta^{k}$;
    \item There exists an arithmetic progression $\mc{P}\subseteq [N]$ of length at least $N^{1/\exp(\log(1/\delta)^{C})}$ such that 
    \[\mb{E}_{n\in\mc{P}}f(n)\ge(1+c')\delta.\]
\end{itemize}
\end{lemma}

We prove Theorem~\ref{thm:main} given Lemma~\ref{lem:dens-increm}; this is the standard density increment strategy.
\begin{proof}[Proof of Theorem~\ref{thm:main} given Lemma~\ref{lem:dens-increm}]
Suppose $A\subseteq [N]$ has no $k$-term arithmetic progressions. We iteratively increase the density of $A$; set $A = A_1$, $N = N_1$ and $\delta = \delta_1$ and we iteratively define $A_i\subseteq [N_i]$, and $\delta_i = |A_i|/N_i$.

If $N_i\le \exp(\exp(\log(1/\delta_i)^{C}))$, we immediately terminate. Otherwise, note that as $A_i$ is free of $k$-term arithmetic progressions, we have that 
\[|\Lambda_k(\mbm{1}_{A_i}) - \Lambda_k(\delta_{i}\cdot \mbm{1}_{ [N_i]})|\ge \delta_i^{k} \cdot |\Lambda_k(\mbm{1}_{[N_i]})| - |A_i|\cdot N_i^{-2} \gg_{k} \delta_i^{k}\]
where we have used that $N_i\ge \exp(\exp(\log(1/\delta_i)^{C}))\gg \delta_i^{-k}$. Therefore, the third case in Lemma~\ref{lem:dens-increm} occurs and there exists $\mc{P}_{i+1}$ such that
\[|A_i\cap \mc{P}_{i+1}|/|\mc{P}_{i+1}|\ge (1+c')\delta_i\]
and $|\mc{P}_{i+1}|\ge N_i^{1/\exp(\log(1/\delta_i)^{C})}$. We now rescale the arithmetic progression $\mc{P}_{i+1}$ to $[|\mc{P}_{i+1}|]=:[N_{i+1}]$, which sends $A_i\cap\mc{P}_{i+1}$ to a new set $A_{i+1}$, and then we continue the iteration.

Note that at every iteration the density $\delta_i$ increases by a multiplicative factor of at least $(1+c')$, so we must terminate in at most $O_k(\log(1/\delta))$ iterations. Thus there exists an index $j\le O_k(\log(1/\delta))$ such that 
\[N^{1/\exp(O_k(\log(1/\delta)^{C+1}))}\le N_j\le \exp(\exp(\log(1/\delta_j)^{C}))\le \exp(\exp(\log(1/\delta)^{C})).\]
This implies that 
\[\log N \le \exp(O_k(\log(1/\delta)^{O_k(1)}))\]
and thus 
\[\delta\le \exp(-(\log\log N)^{\Omega_k(1)}). \qedhere\]
\end{proof}

In order to prove Lemma~\ref{lem:dens-increm}, we first iterate Theorem~\ref{thm:main-inv} to obtain the following result. 
\begin{lemma}\label{lem:factor-iterate}
Fix a parameter $\eta\in(0,1/2)$ and $k\ge 5$. There exists a constant $C = C_{k}>0$ such that the following statement holds. If $N\ge \exp(\log(1/\eta)^{C})$ and $f\colon[N]\to\mb{R}$ is $1$-bounded then there exist functions $h_1,\ldots,h_T\colon[N]\to\mb{R}$ and $d,M,K\ge 1$ such that:
\begin{itemize}
    \item $\mc{B} = \bigvee_{1\le i\le T}\mc{B}_{h_i,K}$ satisfies $\snorm{f- \Pi_{\mc{B}}f}_{U^{k-1}[N]}\le\eta$;
    \item $T, M, K \le \exp(\log(1/\eta)^{C})$ and $d\le\log(1/\eta)^C$;
    \item $h_i = F_i(g_i(n)\Gamma_i)$ is a nilsequence where $g_i(n)$ takes values in a group $G_i$ which is given a degree $(k-2)$ filtration, $G_i/\Gamma_i$ has complexity bounded by $M$ and dimension bounded by $d$, and $F_i\colon G_i/\Gamma_i\to \mb{R}$ is $M$-Lipschitz;
    \item $\mc{B}_{h_i,K}$ is $C$-regular for $1\le i\le T$.
\end{itemize}
\end{lemma}
\begin{proof}
The proof follows via applying Theorem~\ref{thm:main-inv} repeatedly. We begin the iteration by setting $\mc{B}_0 = [N]$ (i.e., the trivial partition). At each stage we will construct $h_{i+1}$ and then set $\mc{B}_{i+1} = \mc{B}_i \vee \mc{B}_{h_{i+1},K}$ with $K = \lceil\exp(\log(1/\eta)^{O_k(1)})\rceil$, where the implicit constant is chosen sufficiently large.

\noindent\textbf{Step 1:} If $\snorm{f-\Pi_{\mc{B}_i}f}_{U^{k-1}[N]}\le\eta$, we terminate.

\noindent\textbf{Step 2:} If $\snorm{f-\Pi_{\mc{B}_i}f}_{U^{k-1}[N]}>\eta$, by Theorem~\ref{thm:main-inv}, there exists a nilsequence $F_{i+1}(g_{i+1}(n)\Gamma_{i+1})$ such that 
\[\big|\mb{E}_{n\in[N]}[(f- \Pi_{\mc{B}_i}f)(n)\ol{F_{i+1}(g_{i+1}(n)\Gamma_{i+1}})]\big|\ge \exp(-\log(1/\eta)^{O_k(1)})\]
and where $G_{i+1}/\Gamma_{i+1}$ has complexity bounded by $\exp(\log(1/\eta)^{O_k(1)})$, dimension bounded by $\log(1/\eta)^{O_k(1)}$, $F_{i+1}\colon G_{i+1}/\Gamma_{i+1}\to\mb{C}$ is $\exp(\log(1/\eta)^{O_k(1)})$-Lipschitz, $G_{i+1}$ has been given a degree $(k-2)$ filtration, and where $g_{i+1}(n)$ is a polynomial sequence with respect to this filtration. Taking either the real or imaginary part of $F_{i+1}$, we may assume that $F_{i+1}\colon G_{i+1}/\Gamma_{i+1}\to\mb{R}$ and thus that 
\[\big|\mb{E}_{n\in[N]}[(f-\Pi_{\mc{B}_i}f)(n) F_{i+1}(g_{i+1}(n)\Gamma_{i+1})]\big|\ge\exp(-\log(1/\eta)^{O_k(1)}).\]
Note that for any $t\in[0,1/K)$, this implies that 
\begin{align*}
\bigg|\mb{E}_{n\in[N]}\bigg[(f- \Pi_{\mc{B}_i}f)(n) &\frac{\lfloor K(F_{i+1}(g_{i+1}(n)\Gamma_{i+1}) + t)\rfloor}{K}\bigg]\bigg| \\
&\ge \big|\mb{E}_{n\in[N]}[(f- \Pi_{\mc{B}_i}f)(n) F_{i+1}(g_{i+1}(n)\Gamma_{i+1})]\big| - 2/K\\
&\ge\exp(-\log(1/\eta)^{O_k(1)})
\end{align*}
given that the implicit constant defining $K$ is chosen sufficiently large. Recall here $\lfloor x\rfloor $ is defined in the standard manner that $x = \lfloor x\rfloor + \{x\}$ where $\lfloor x\rfloor \in \mb{Z}$ and $\{x\}\in [0,1)$. We then take $t\in [0,1/K)$, such that $\mc{B}_{F_{i+1}(g_{i+1}(n)\Gamma_{i+1}) + t,K}$ is $C$-regular; this exists for $C$ larger than an absolute constant by Fact~\ref{fct:maximal}. 

Set $h_{i+1}(n):=F_{i+1}(g_{i+1}(n)\Gamma_{i+1}) + t$. Note that
\[\frac{\lfloor K(F_{i+1}(g_{i+1}(n)\Gamma_{i+1}) + t)\rfloor}{K}\]
is measurable with respect to $\mc{B}_{h_{i+1},K}$ by construction and it is bounded by $\exp(\log(1/\eta)^{O_k(1)})$. Therefore since $\Pi_{\mc{B}_{h_{i+1},K}}$ is self-adjoint we have 
\begin{align*}
&\mb{E}_{n\in[N]}[|\Pi_{\mc{B}_{h_{i+1},K}}(f- \Pi_{\mc{B}_i}f)(n)|]\\
&\qquad\quad\ge(1+\snorm{F_{i+1}}_{L^\infty(G_{i+1}/\Gamma_{i+1})})^{-1}\cdot \bigg|\mb{E}_{n\in[N]}\bigg[(f- \Pi_{\mc{B}_i}f)(n)\frac{\lfloor K(F_{i+1}(g_{i+1}(n)\Gamma_{i+1}) + t)\rfloor}{K}\bigg]\bigg|\\
&\qquad\quad \ge \exp(-\log(1/\eta)^{O_k(1)}).
\end{align*}
    
\noindent\textbf{Step 3:} We now return back to Step 1 and keep on iterating this procedure until it terminates. This completes the proof modulo showing that the iteration terminates in a small number of steps. To show this, note that
\begin{align*}
\snorm{\Pi_{\mc{B}_{h_{i+1},K}}(f-\Pi_{\mc{B}_i}f)}_{L^1[N]}&\le\snorm{\Pi_{\mc{B}_{h_{i+1},K}}(f-\Pi_{\mc{B}_i}f)}_{L^2[N]}=\snorm{\Pi_{\mc{B}_{h_{i+1},K}}\Pi_{\mc{B}_{i+1}}(f-\Pi_{\mc{B}_i}f)}_{L^2[N]}\\
&\le\snorm{\Pi_{\mc{B}_{i+1}}(f-\Pi_{\mc{B}_i}f)}_{L^2[N]}=\snorm{\Pi_{\mc{B}_{i+1}}f-\Pi_{\mc{B}_i}f}_{L^2[N]}\\
&=(\snorm{\Pi_{\mc{B}_{i+1}}f}_{L^2[N]}^2-\snorm{\Pi_{\mc{B}_i}f}_{L^2[N]}^2)^{1/2}.
\end{align*}
The final equality is the Pythagorean theorem with respect to projections (this follows from e.g.~\cite[Lemma~4.3(iv)]{PP22}). We deduce
\[\snorm{\Pi_{\mc{B}_{i+1}}f}_{L^2[N]}^2-\snorm{\Pi_{\mc{B}_i}f}_{L^2[N]}^2\ge \exp(-\log(1/\eta)^{O_k(1)}).\]
Since for all $i$ we have $\snorm{\Pi_{\mc{B}_i}f}_{L^2[N]}\le \snorm{f}_{L^2[N]}\le 1$, there are at most $\exp(\log(1/\eta)^{O_k(1)})$ iterations as desired. 
\end{proof}

We now complete the proof of Lemma~\ref{lem:dens-increm} and therefore the proof of Theorem~\ref{thm:main}. The first part of the proof is finding a density increment on a factor derived from nilsequences, which is essentially identical to that of \cite[Lemma~5.8]{GT09}. In the second part, we apply our nilsequence Schmidt-type result Lemma~\ref{lem:schmidt-nilman} to find a long arithmetic progression with density increment.
\begin{proof}[Proof of Lemma~\ref{lem:dens-increm}]
Without loss of generality, we may assume that $c$ is smaller than an absolute constant. Furthermore we may assume that $N\ge \exp(\exp(\log(1/\delta)^{\Omega(1)}))$ (where the implicit constant may depend on $c,k$) and $|\Lambda_k(f) - \Lambda_k(\delta \cdot \mbm{1}_{[N]})|\ge c\delta^{k}$.

\noindent\textbf{Step 1: Increment on a factor.}
By applying Lemma~\ref{lem:factor-iterate}, there exists a factor $\mc{B}$ (derived from nilsequences of appropriate complexity, with parameters below) such that
\[\snorm{\Pi_{\mc{B}}f - f}_{U^{k-1}[N]}\le c^\ast\delta^k\]
where we choose $c^\ast$ sufficiently small in terms of $c$. Via telescoping and the second inequality in Lemma~\ref{lem:basic-fct}, we have 
\[|\Lambda_k(f) - \Lambda_k(\Pi_{\mc{B}}f)|\le c\delta^k/2\]
as long as $c^\ast$ was chosen appropriately, and therefore 
\[|\Lambda_k(\Pi_{\mc{B}}f) - \Lambda_k(\delta \cdot \mbm{1}_{[N]})|\ge c\delta^k/2.\]

Take $c' = \min(c,1)/(10k)^5$. Let $g = \min(\Pi_{\mc{B}}f, (1+c')\delta)$. The crucial claim is that if $\Omega'=\{n\in[N]\colon g(n)\neq\Pi_{\mc{B}}f(n)\} = \{n\in[N]\colon\Pi_{\mc{B}}f(n)>(1+c')\delta\}$ then $\Omega'$ must have sufficiently large measure. To see this note that:
\begin{align*}
|\Lambda_k(\Pi_{\mc{B}}f) - \Lambda_k(g)|&\le k\snorm{\Pi_{\mc{B}}f - g}_{L^1[N]} \le k \mb{P}_{n\in[N]}[n\in \Omega'],\\
|\Lambda_k(\delta\mbm{1}_{[N]}) - \Lambda_k(g)|&\le k(1+c')^{k-1}\delta^{k-1}\snorm{\delta\mbm{1}_{[N]}-g}_{L^1[N]},\\
\snorm{g - \delta \mbm{1}_{[N]}}_{L^1[N]}&\le \mb{P}_{n\in[N]}[n\in \Omega'] + \snorm{\delta \mbm{1}_{[N]} - \Pi_{\mc{B}}f}_{L^1[N]}.
\end{align*}
The first and second inequality follow from the first part of Lemma~\ref{lem:basic-fct} and telescoping while the final inequality follows from the triangle inequality. We simplify the inequalities slightly; as $\mb{E}_{n\in[N]}[\delta \mbm{1}_{[N]}] = \mb{E}_{n\in[N]}[f]=\mb{E}_{n\in[N]}[\Pi_{\mc{B}}f]$, we have 
\begin{align*}
\snorm{\delta \mbm{1}_{[N]} - \Pi_{\mc{B}}f}_{L^1[N]} &= 2\snorm{\max(\Pi_{\mc{B}}f - \delta \mbm{1}_{[N]},0)}_{L^1[N]}\le 2c'\delta + 2\mb{P}_{n\in[N]}[n\in \Omega'].
\end{align*}
Given this and using the upper bound on $c'$, we deduce 
\begin{align*}
|\Lambda_k(\Pi_{\mc{B}}f) - \Lambda_k(g)|&\le k\snorm{\Pi_{\mc{B}}f - g}_{L^1[N]} \le k \mb{P}_{n\in[N]}[n\in\Omega'],\\
|\Lambda_k(\delta\mbm{1}_{[N]}) - \Lambda_k(g)|&\le 2k\delta^{k-1}\snorm{\delta\mbm{1}_{[N]}-g}_{L^1[N]},\\
\snorm{g - \delta \mbm{1}_{[N]}}_{L^1[N]}&\le 3\mb{P}_{n\in[N]}[n\in \Omega'] + 2c'\delta.
\end{align*}
Therefore 
\begin{align*}
c\delta^{k}/2 &\le |\Lambda_k(\delta\cdot \mbm{1}_{[N]}) - \Lambda_k(\Pi_{\mc{B}}f)|\le |\Lambda_k(\delta\cdot \mbm{1}_{[N]}) - \Lambda_k(g)| + |\Lambda_k(\Pi_{\mc{B}}f) - \Lambda_k(g)| \\
&\le k \mb{P}_{n\in[N]}[n\in \Omega'] + 2k\delta^{k-1}\snorm{\delta \mbm{1}_{[N]} - g}_{L^{1}[N]}\le 7k\mb{P}_{n\in[N]}[n\in \Omega'] + 4kc'\delta^{k};
\end{align*}
thus we have $\mb{P}_{n\in[N]}[n\in \Omega']\ge c\delta^{k}/(20k)$.

\noindent\textbf{Step 2: Increment on a progression.}
We are now in position to apply the nilsequence Schmidt-type result Lemma~\ref{lem:schmidt-nilman}. Recall that we applied Lemma~\ref{lem:factor-iterate} to find $\mc{B}$, and hence we may write $\mc{B} = \bigvee_{1\le i\le T}\mc{B}_{h_i,K}$ where:
\begin{itemize}
    \item $T, M, K \le\exp(\log(1/\delta)^C)$ and $d\le\log(1/\delta)^{C}$;
    \item $h_i = F_i(g_i(n)\Gamma_i)$ is a nilsequence where $g_i(n)$ takes values in a group $G_i$ which is given a degree $(k-2)$ filtration, $G_i/\Gamma_i$ has complexity bounded by $M$ and dimension bounded by $d$, and $F_i\colon G_i/\Gamma_i\to\mb{R}$ is $M$-Lipschitz;
    \item $\mc{B}_{h_i,K}$ is $C$-regular for $1\le i\le T$.
\end{itemize}
Here $C$ is a slightly larger value than the constant $C_k$ in Lemma~\ref{lem:factor-iterate}, depending only on $k$.

We now apply Lemma~\ref{lem:schmidt-nilman} to $g_i(n)$ for $1\le i\le T$. We obtain a decomposition of $[N]$ into arithmetic progressions $\mc{P}_1,\ldots,\mc{P}_L$ such that
\begin{itemize}
    \item $N/L\ge N^{-1/\exp(\log(1/\delta)^{O_k(1)})}$;
    \item We have
    \[\max_{\substack{1\le i\le T\\1\le j\le L}}\max_{n,n'\in \mc{P}_j}d_{G_i/\Gamma_i}(g_i(n)\Gamma_i,g_i(n')\Gamma_i)\le \exp(\log(1/\delta)^{O_k(1)})\cdot N^{-1/\exp(\log(1/\delta)^{O_k(1)})}.\]
\end{itemize}

We now consider $\mc{P}_j$ which intersect $\Omega'$. Call a progression in the decomposition \emph{crossing} if it intersects $\Omega'$ and $[N]\setminus\Omega'$ and a progression \emph{contained} if it is fully within in $\Omega'$. Since $\Omega'$ is measurable in terms of $\mc{B}$, for a progression to be crossing it must ``cross a boundary'' defining $\mc{B}_{h_i,K}$ for at least one $1\le i\le T$. If a progression $\mc{P}_j$ crosses one of these boundaries defined by $h_i$ then all points in $\mc{P}_j$ map close to this boundary, since the function $F_i$ is $M$-Lipschitz. In particular, by regularity of each $\mc{B}_{h_i,K}$, the measure (with respect to the uniform distribution on $[N]$) of improper progressions is bounded by
\[\ll_k T\cdot\exp(\log(1/\delta)^{O_k(1)})\cdot N^{-1/\exp(\log(1/\delta)^{O_k(1)})} = \exp(\log(1/\delta)^{O_k(1)})\cdot N^{-1/\exp(\log(1/\delta)^{O_k(1)})}.\]

Let $\Omega^\ast$ denote the union of all the contained progressions which have length at least $N'=cc'\delta^{k + 1}/(400k) \cdot N/L$ (hence certainly $\Omega^\ast\subseteq\Omega'$). Let $\mc{I}$ be the set of all $1\le i\le L$ so that $\mc{P}_i$ either has length at most $N'$ or is crossing.  We easily see that 
\begin{align*}
0&\le\mb{P}_{n\in[N]}[n\in\Omega']-\mb{P}_{n\in[N]}[n\in\Omega^\ast]\le\sum_{i\in\mc{I}}\mb{P}_{n\in[N]}[n\in\mc{P}_i]\le cc'\delta^{k+1}/(200k);
\end{align*}
in the final inequality we have used that $N\ge\exp(\log(1/\delta)^{\Omega(1)})$ for a sufficiently large implicit constant.

Finally, this implies that
\begin{align*}
\mb{E}_{n\in \Omega^{\ast}}[f] &= \frac{\mb{E}_{n\in[N]}[f \cdot \mbm{1}_{n\in \Omega^{\ast}}]}{\mb{P}_{n\in[N]}[n\in \Omega^{\ast}]}\ge \frac{\mb{E}_{n\in[N]}[f \cdot \mbm{1}_{n\in \Omega'}] - cc'\delta^{k+1}/(200k)}{\mb{P}_{n\in[N]}[n\in \Omega']}\\
&\ge \frac{\mb{E}_{n\in[N]}[f\cdot \mbm{1}_{n\in \Omega'}]}{\mb{P}_{n\in[N]}[n\in \Omega']} - \frac{cc'\delta^{k+1}/(200k)}{c\delta^{k}/(20k)}\\
&= \frac{\mb{E}_{n\in[N]}[\Pi_{\mc{B}}f\cdot \mbm{1}_{n\in \Omega'}]}{\mb{P}_{n\in[N]}[n\in \Omega']} - c'\delta/10\\
&\ge (1+c')\delta - c'\delta/10\ge (1 + c'/2)\delta. 
\end{align*}
By pigeonhole, this implies that there exists a contained arithmetic progression $\mc{P}_i$ having length at least $cc'\delta^{k + 1}/(400k) \cdot N/L \ge N^{-1/\exp(\log(1/\delta)^{O_k(1)})}$ on which the density of $f$ is at least $(1+c'/2)\delta$. Adjusting the value of $c'$, this completes the proof.
\end{proof}

\bibliographystyle{amsplain1}
\bibliography{main.bib}

\end{document}